\documentclass{article}

\usepackage{setspace}
\usepackage{color}
\usepackage{amsfonts}
\usepackage{amsmath}
\usepackage{amsthm}
\usepackage{amssymb}
\usepackage{graphicx}

\newcommand{\br}{\mathbb R}

\newcommand{\nc}{\newcommand}
\newtheorem{lemma}{Lemma}[section]
\newtheorem{theorem}{Theorem}[section]
\newtheorem{problem}{Problem}[section]
\newtheorem{definition}{Definition}[section]

\nc{\cD}{{\cal D}}
\nc{\cP}{{\cal P}}
\nc{\cR}{{\cal R}}
\nc{\e}{\varepsilon}
\nc{\Om}{\Omega}
\nc{\om}{\omega}
\nc{\aal}{\alpha}
\nc{\tow}{\rightharpoonup}
\nc{\nin}{\in \hs{-.35}/\,}
\nc{\np}{\newpage}
\nc{\g}{\gamma}
\nc{\IN}{I \hs{-.15} N}
\nc{\IR}{I \hs{-.14} R}
\nc{\IK}{I \hs{-.14} K}
\nc{\hs}[1]{\hspace{#1cm}}
\def\C#1{{\mathcal {#1}}}
\title{\bf Attractors for Navier-Stokes flows \\with multivalued and nonmonotone subdifferential boundary conditions}

\author{Piotr Kalita
    \thanks{E-mail : piotr.kalita@ii.uj.edu.pl}
    \thanks{The research was supported by the Marie Curie International Research Staff
Exchange Scheme Fellowship within the 7th European Community Framework
Programme under Grant Agreement No. 295118, by the National Science
Center of Poland under grant no. N N201 604640, and by the International
Project co-financed by the Ministry of Science and Higher Education of
Republic of Poland under grant no. W111/7.PR/2012.},
Grzegorz {\L}ukaszewicz
    \thanks{E-mail : glukasz@mimuw.edu.pl, Tel.: +48 22 55 44 562}
\thanks{This research was supported by Polish Government Grant N N201 547638} }

{\large\date{}}

\begin{document}
\maketitle
\begin{center}

 {\small
$^{\star\dag}$Faculty of Mathematics and Computer Science,
Institute of Computer Science,
Jagiellonian University,
ul. prof. S. {\L}ojasiewicza 6, 30-348 Krak\'ow, Poland,

$^{\ddag\S}$University of Warsaw, Mathematics Department,
     ul.Banacha 2, 02-957 Warsaw, Poland}
\end{center}
\normalsize

\begin{abstract}
\noindent We consider two-dimensional nonstationary Navier-Stokes shear flow with multivalued and nonmonotone boundary conditions on a part of the boundary of the flow domain. We prove the existence of global in time solutions of the considered problem which is governed by a partial
differential inclusion with a multivalued term in the form of Clarke subdifferential.
Then we prove the existence of a trajectory attractor and a weak global attractor for the associated multivalued semiflow.

This research is motivated by control problems for fluid flows in domains with semipermeable walls and membranes.
\end{abstract}

\vspace{0.2cm}


\vspace{0.2cm} \noindent{\it 1991 Mathematics Subject
Classification:} 76D05, 76F10, 76F20, 47J20, 49J40 \vspace{0.2cm}
 \renewcommand{\theequation}{\arabic{section}.\arabic{equation}}
 \setcounter{equation}{0}
 \section{Introduction}

In this paper we consider two-dimensional nonstationary incompressible Navier-Stokes shear flows with nonmonotone boundary conditions on a part of the boundary of the flow domain. Our aim is to prove the existence of global in time solutions of the considered problem which is governed by a partial differential inclusion, and then to prove the existence of a trajectory attractor and a weak global attractor for the associated multivalued semiflow.

This research is motivated by control problems for fluid flows in domains with semipermeable walls and membranes.

The problem we consider is as follows. The flow of an incompressible fluid in a two-dimensional domain $\Omega$
is described by the equation of motion
\begin{equation} \label{e2.1}
 u_{t} -\nu\Delta u + (u\cdot\nabla) u +\nabla p = 0 \quad {\rm for} \quad (x,t)\in\Omega\times\br^+
\end{equation}
\noindent and the incompressibility condition
\begin{equation} \label{e2.2}
{\rm div} \, u = 0 \quad {\rm for} \quad (x,t)\in\Omega\times\br^+.
\end{equation}
To define the domain $\Omega$ of the flow let us consider the channel
$$\Omega_{\infty} =\{x=(x_1,x_2): -\infty < x_1 < \infty,\,\,\,
          0 < x_2 < h(x_{1})\}, $$
where the function $h:\br\to\br$ is a positive, smooth, and L-periodic.
Then we set
  $$\Omega =   \{x=(x_1,x_2): 0 < x_1 < L, \,\,\, 0 < x_2 < h(x_{1})\} $$
and
  $\partial\Omega=\bar{\Gamma}_{0}\cup\bar{\Gamma}_{L}\cup\bar{\Gamma}_{1}$,
where $\Gamma_{0}$ and $\Gamma_{1}$ are the bottom and the top,
and $\Gamma_{L}$ is the lateral part of the boundary of $\Omega$. The domain $\Omega$
is schematically presented in Fig. \ref{fig:domain}.

\begin{figure}[!h]
\begin{center}
\includegraphics[width=6cm]{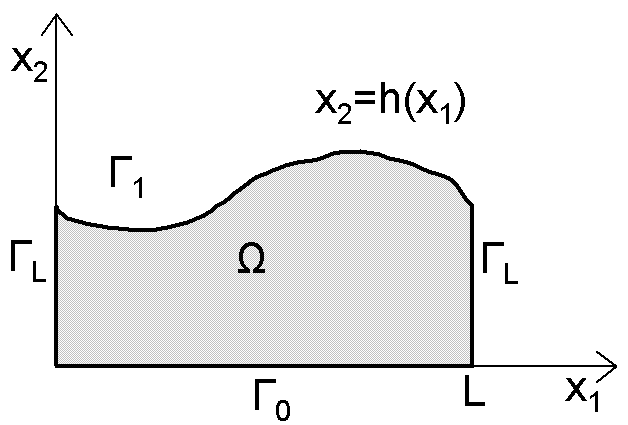}
\end{center}
\caption{Schematical view of $\Omega$.} \label{fig:domain}
\end{figure}

We are interested in solutions of (\ref{e2.1})-(\ref{e2.2})
in $\Omega\times\br^+$ which are L-periodic with respect to $x_1$.
 We assume that
\begin{equation} \label{e2.3}
 u=0 \quad  {\rm at} \quad \Gamma_{1}\times \br^+.
 \end{equation}

\noindent On the bottom $\Gamma_{0}$ we impose the following conditions.
The tangential component $u_T$ of the velocity vector on $\Gamma_{0}$ is given, namely, for some $s\in \br$,
\begin{equation} \label{e2.4}
u_T = u-u_Nn = (s,0) \quad  {\rm at} \quad \Gamma_{0}\times\br^+, \quad  {\rm where} \quad u_N=u\cdot n.
\end{equation}
Furthermore, we assume the following subdifferential boundary condition
\begin{equation} \label{e2.4a}
\tilde{p}(x,t) \in \partial j(u_N(x,t))  \quad {\rm at} \quad \Gamma_{0}\times \br^+,
\end{equation}
where $\tilde{p}=p+\frac{1}{2}|u|^2$ is the total pressure (called also the Bernoulli pressure),  $j:\br\to\br$ is a given locally Lipschitz superpotential, and  $\partial j$ is a Clarke subdifferential of $j(\cdot)$ (see for example \cite{Clarke}, \cite{DMP1} for the definition and properties of Clarke subdifferential).

Let, moreover,
\begin{eqnarray} \label{initial1}
      u(0)= u_0\quad {\rm in}\quad\Omega.
\end{eqnarray}
The considered problem is motivated by the examination of a certain two-dimensional flow in an infinite
(rectified) journal bearing $\Omega\times(-\infty,+\infty)$, where
$\Gamma_1\times(-\infty,+\infty)$ represents the outer cylinder,
and $\Gamma_0\times(-\infty,+\infty)$ represents the inner,
rotating cylinder. In the lubrication problems the gap $h$ between
cylinders is never constant. We can assume that the rectification
does not change the equations as the gap between cylinders is very
small with respect to their radii.

A physical interpretation of the boundary condition (\ref{e2.4a}) can be as follows. The superpotential $j$ in our control problem is not convex as it corresponds to the {\it nonmonotone} relation between the normal velocity $u_N$ and the total pressure $\tilde{p}$ at $\Gamma_0$. Assuming that, left uncontrolled, the total pressure at $\Gamma_0$ would increase with the increase of the normal velocity of the fluid at $\Gamma_0$, we control $\tilde{p}$ by a hydraulic device which opens wider the boundary orifices at $\Gamma_0$ when $u_N$ attains a certain value and thus $\tilde{p}$ {\it drops} at this value of $u_N$. Particular examples of such relations are provided in \cite{Migorski-Ochal2007} and \cite{Migorski-Ochal2005}.

The knowledge or the judicious choice of the boundary conditions on the fluid-solid interface is of particular interest in lubrication area which is concerned with thin film flow behaviour. The boundary conditions to be employed are determined by numerous physical parameters characterizing, for example, surface roughness and rheological properties of the fluid.

The system of equations (\ref{e2.1})-(\ref{e2.2}) with {\it non-slip} boundary conditions: (\ref{e2.3}) at $\Gamma_1$ for $h=const$ and $u=const$ on $\Gamma_{0}$ (instead of (\ref{e2.4})-(\ref{e2.4a}) on $\Gamma_0$) was intensively studied in
several contexts, some of them mentioned in the introduction of \cite{mb-gl-cambridge-2009}.
The autonomous case with $h\neq const$ and with $u=const$ on $\Gamma_{0}$ was considered in \cite{BoLu2-04, BoLu1-04} and  the nonautonomous case $h\neq const$,
$u=U(t)e_{1}$ on $\Gamma_{0}$ was considered in \cite{BoLuR-05}. Existence of exponential attractors for the Navier-Stokes and Bingham fluids with the Tresca boundary condition on $\Gamma_{0}$ was proved in \cite{gl-2012-RWA,gl-2012-DCDS}. Recently, attractors for two dimensional Navier Stokes flows with Dirichlet boundary conditions were studied in \cite{Coti}, where, in contrast to this paper, the time continuous problem has a unique solution and theory of multivalued flows is needed to study the time discretized systems.

Asymptotic behaviour of solutions for the problems governed by partial differential inclusions where the multivalued term has the form of Clarke subdifferential was studied in \cite{Kasyanov2012} and \cite{Kasyanov2013}, where the reaction-diffusion problem with multivalued semilinear term was considered, and in \cite{Kalita2010}, where the strongly damped wave equation with multivalued boundary conditions was analyzed.

For the problem considered in this paper, existence of weak solutions for the case $u_T=0$ in place of (\ref{e2.4}) was shown in \cite{Migorski-Ochal2007}.

Note that due to nonmonotone and multivalued boundary condition (\ref{e2.4a}) the formulated problem can have multiple weak solutions. The main tool used in this paper to prove existence of attractors is the theory of trajectory attractors, which, instead of the direct analysis of the multivalued semiflow (i.e. map that assigns to initial condition the set of states obtainable after some time $t$), focuses on the shift operator defined on the space of trajectories for the studied problem. This approach was introduced in papers \cite{chep-vish-1996}, \cite{malek-necas-1996} and \cite{Sell} as a method to avoid the nonuniqueness of solutions, indeed, the shift operator is \textit{uniquely} defined even if the dynamics of the problem is governed by the \textit{multivalued} semiflow. Recent results and open problems in the theory of trajectory attractors are discussed in the survey papers \cite{balibrea} and \cite{chep-vish-2011}.

The plan of the paper is as follows. In Section~\ref{secvf} we give a variational formulation of the problem. In Section~\ref{existence-sol} we prove the existence of global in time solutions, and in Section~\ref{trajectory-attractor} we prove the existence of a trajectory attractor and a weak global attractor.
\renewcommand{\theequation}{\arabic{section}.\arabic{equation}}
\setcounter{equation}{0}
\section{Weak formulation of the problem}\label{secvf}
In this section we introduce the basic notations and define a notion of a weak solution $u$ of the initial boundary value problem (\ref{e2.1})-(\ref{initial1}). For convenience, we shall work with a homogenized problem whose solution $v$ has the tangential component $v_T$ at $\Gamma_0$ equal to zero, and then $u=v+w$ for a suitable extention $w$ of the boundary data.

In order to define a weak formulation of the homogenized problem (\ref{e2.1})-(\ref{initial1}) we need to introduce some function spaces and operators.

Let
\begin{eqnarray*}
W=\{w\in C^\infty(\bar{\Omega};\br^2):\, {\rm div}\ w=0 \,\, {\rm in } \,\,\Omega, \,\,
w\,\,{\rm is \,\,L-periodic\,\,in }\,\,x_1,\,\,w&=&0\,\,{\rm at }\,\,\Gamma_1,\\ w_T&=&0\, \, {\rm on} \,\, \Gamma_0 \},
\end{eqnarray*}
and let $V$ and $H$ be the closures of $W$ in the norms of $H^1(\Omega,\br^2)$ and $L^2(\Omega,\br^2)$, respectively. In the sequel we will use notation $\|\cdot\|, \|\cdot\|_H$ to denote, respectively, norms in $V$ and $H$. We denote the trace operator $V\to L^2(\Gamma_0;\br^2)$ by $\gamma$. By the trace theorem, $\gamma$ is linear and bounded; we will denote its norm by $\|\gamma\|:=\|\gamma\|_{\C{L}(V;L^2(\Gamma_0;\br^2))}$. In the sequel we will write $u$ instead of $\gamma u$ for the sake of notation simplicity.

Let the operators $A:H^1(\Omega,\br^2)\to V^\star$ and $B[\cdot]:H^1(\Omega,\br^2)\to V^\star$ be defined by
\begin{eqnarray}\label{defA}
     \langle Au, v \rangle = \nu\int_\Omega \mathrm{rot}\ u\cdot\mathrm{rot}\ v \, dx
\end{eqnarray}
for all $u\in H^1(\Omega,\br^2)$, $v\in V$, and $B[u]=B(u,u)$, where
\begin{eqnarray} \label{defB}
     \langle B(u,w), z \rangle = \int_\Omega (\mathrm{rot}\ u\times w)\cdot z \, dx
\end{eqnarray}
for $u, w \in H^1(\Omega,\br^2)$, $z \in V$.

According to the hydrodynamical interpretation of the considered problem given in the Introduction, we can understand the $\mathrm{rot}$ operators as follows. For       $u(x_1,x_2)=(u_1(x_1,x_2),u_2(x_1,x_2))$, $\bar{x}=(x_1,x_2,x_3)$, and
    $\bar{u}(\bar{x})= (u_1(x_1,x_2),u_2(x_1,x_2),0)$, we set
    $\mathrm{rot}\,u(x_1,x_2)=\mathrm{rot}\,\bar{u}(\bar{x})$.

Let $G=\Omega\times (0,1)$ and $f(\bar{x})=f(x_1,x_2)$ be a scalar function. Then
\begin{eqnarray}\label{eqnpom}
     \int_\Omega f(x_1,x_2) dx = \int_G f(\bar{x })\,d\bar{x}.
\end{eqnarray}
In particular, for $u,v$ in $V$,
\begin{eqnarray}\label{eqn-rot-nabla}
     \langle Au, v \rangle &=& \nu\int_\Omega \mathrm{rot}\,u(x)\cdot\mathrm{rot}\,v(x) \, dx =
     \nu\int_G \mathrm{rot}\,\bar{u}(\bar{x})\cdot\mathrm{rot}\,\bar{v}(\bar{x})\, d\bar{x} \nonumber\\
     &=& \nu\int_G \nabla\bar{u}(\bar{x})\cdot\nabla\bar{v}(\bar{x})\, d\bar{x} =
     \nu\int_\Omega\nabla u(x)\cdot\nabla v(x)\,dx.
\end{eqnarray}
To work with the boundary condition (\ref{e2.4a}) we rewrite equation of motion (\ref{e2.1}) in the Lamb form,
\begin{equation} \label{e2.1L}
 u_{t} +\nu\, \mathrm{rot}\, \mathrm{rot}\, u + \mathrm{rot}\, u \times u +\nabla \tilde{p} = 0 \quad {\rm in} \quad \Omega\times\br^+.
\end{equation}
Further, to homogenize the problem, for $u\in H^1(\Omega,\br^2)$ let $w\in H^1(\Omega,\br^2)$ be such that $w_T=u_T=s$ and $w_N=0$ on $\Gamma_0$, and let $u=v+w$. Then $v\in V$ as $v_T=0$ at $\Gamma_0$. Moreover,
$v_N=u_N$ on $\Gamma_0$.

Multiplying (\ref{e2.1L}) by $z\in V$ and using the Green formula we obtain
\begin{eqnarray} \label{e33}
      \langle v'(t) + Av(t) + B[v(t)],z \rangle + \int_{\Gamma_0} \tilde{p}z_N\,d\sigma(x) = \langle F, z \rangle
      + \langle G(v), z \rangle,
\end{eqnarray}
where
\begin{eqnarray} \label{righ1}
     \langle F, z \rangle =  \nu\int_\Omega \mathrm{rot}\,w \cdot \mathrm{rot}\,z\, dx - \langle B(w,w), z \rangle
\end{eqnarray}
and
\begin{eqnarray}
     \langle G(v), z \rangle = - \langle B(v,w) + B(w,v), z \rangle
\end{eqnarray}

Above we have used the formula
\begin{eqnarray} \label{general-formula}
    \int_\Omega \mathrm{rot}\,R \cdot a\,dx = \int_\Omega R \cdot \mathrm{rot}\,a\, dx +
    \int_{\partial\Omega}(R\times a)\cdot n\, d\sigma(x)
\end{eqnarray}
with  $R=\mathrm{rot}\,v$ or $R=\mathrm{rot}\,w$. Formula (\ref{general-formula}) is easy to get using the three-dimensional vector calculus and (\ref{eqnpom}).
Observe that if $a_T=0$ on $\partial\Omega$, then we have
    $(R\times a)\cdot n = (R\times a_Nn)\cdot n =0$.

We need the following assumptions on the potential $j$:
\begin{itemize}\item[$H(j)$:]
\begin{itemize}
\item[(a)] $j:\br\to\br$ is locally Lipschitz,
\item[(b)] $\partial j$ satisfies the growth condition $|\xi|\leq c_1+c_2|u|$ for all $u\in\br$ and all $\xi\in \partial j(u)$, with $c_1>0$ and $c_2>0$,
\item[(c)] $\partial j$ satisfies the dissipativity condition $\inf_{\xi\in \partial j(u)}\xi u \geq d_1-d_2|u|^2$, for all $u\in \br$ where $d_1\in \br$ and
$d_2\in \left(0,\frac{\nu}{\|\gamma\|^2}\right)$.
\end{itemize}
\end{itemize}
Observe that assumptions $H(j)$ presented here are more general then the corresponding assumptions of Theorem 1 in \cite{Migorski-Ochal2007}, save for the fact that $j$ is assumed
there to depend on space and time variables directly.

From (\ref{e33}) we obtain the following weak formulation of the homogenized problem.

\begin{problem} \label{problem_main}
Let $v_0\in H$. Find $v\in L^2_{loc}(\br^+;V)\cap L^\infty_{loc}(\br^+;H)$
with $v'\in L^{\frac{4}{3}}_{loc}(\br^+;V^\star)$, $v(0)=v_0$,
and such that
\begin{eqnarray} \label{def-sol}
   &&  \langle v'(t) + Av(t) + B[v(t)],z \rangle + (\xi(t),z_N)_{L^2(\Gamma_0)}
     = \langle F, z \rangle +  \langle G(v(t)), z \rangle,\\
   &&  \xi(t)\in S^2_{\partial j(v_N(\cdot,t))}, \nonumber
\end{eqnarray}
for a.e. $t\in\br^+$ and for all $z\in V$.
\end{problem}
In the above definition we use the notation
$$S^2_U=\{u\in L^2(\Gamma_0): u(x)\in U(x)\ \mbox{for a.e.}\ x\in \Gamma_0\},$$
valid for a multifunction $U:\Gamma_0\to 2^{\br}$.
Note that if $v\in L^2_{loc}(\br^+;V)$ and $v'\in L^{\frac{4}{3}}_{loc}(\br^+;V^\star)$
then $v\in C(\br^+;V^\star)$, hence the initial condition make sense. Moreover, since $v\in L^\infty_{loc}(\br^+;H)$ then (see for example \cite{chep-vish-2011}, Theorem II.1.7)
$v\in C_w(\br^+;H)$, i.e. $v$ is weakly continuous as a function of time with values in $H$, and thus the initial condition make sense in the phase space $H$.

One can see (cf. \cite{Migorski-Ochal2007}) that if $v\in L^2_{loc}(\br^+;V)$ is a sufficiently smooth solution of the partial differential inclusion (\ref{def-sol}) then there exists a distribution $\tilde{p}$ such that the conditions (\ref{e2.1L}) and (\ref{e2.4a}) hold for $u=v+w$. In conclusion, the function $u=v+w$ can be regarded as a weak solution of the initial boundary value problem (\ref{e2.1})-(\ref{initial1}), provided  $v$ is a solution of Problem \ref{problem_main} with $v(0)=u_0-w$, $u_0\in H$.

The trajectory space $\mathcal{K}^+$ of Problem \ref{problem_main} is defined as the set of those of its solutions with some $v_0\in H$
that satisfy the following inequality
\begin{equation}\label{eqn:estimate_v}
\frac{1}{2}\frac{d}{dt}\|v(t)\|_H^2+C_1\|v(t)\|^2
     \leq C_2(1+\|F\|^2_{V^*}),
\end{equation}
where $C_1, C_2>0$ and inequality (\ref{eqn:estimate_v}) is understood in the sense, that for all $0\leq t_1 < t_2$ and for all $\psi\in C^\infty_0(t_1,t_2)$, $\psi\geq 0$ we have
\begin{equation}\label{eq:estimate_v_meaning}
-\frac{1}{2}\int_{t_1}^{t_2}\|v(t)\|_H^2\psi'(t)\ dt + C_1\int_{t_1}^{t_2}\|v(t)\|^2\psi(t)\ dt\leq C_2(1+\|F\|^2_{V^*})\int_{t_1}^{t_2}\psi(t)\ dt
\end{equation}
Note, that since we cannot guarantee that for every solution of Problem \ref{problem_main} we have $\langle v'(t),v(t)\rangle = \frac{1}{2}\frac{d}{dt}\|v(t)\|_H^2$, we cannot derive (\ref{eqn:estimate_v}) for every solution of Problem \ref{problem_main}.

In the next section we prove that the trajectory space is not empty.
\renewcommand{\theequation}{\arabic{section}.\arabic{equation}}
\setcounter{equation}{0}
\section{Existence of global in time solutions} \label{existence-sol}

In this section we give the proof of the existence of solutions of Problem \ref{problem_main} that satisfy inequality (\ref{eqn:estimate_v}). The proof will be based on the standard technique that uses the regularization of the multivalued term and in main points will follow  \cite{gl-2012-RWA} and \cite{Migorski-Ochal2007}.

The operators $A$ and $B$ defined in (\ref{defA}) and (\ref{defB}) and restricted to $V$ have the following properties:
\begin{itemize}
\item[(1)] $A:V\to V^*$ is a linear, continuous, symmetric operator such that
\begin{equation}\label{property:A}
\langle Av,v\rangle = \nu\|v\|^2\ \mbox{for}\ v\in V.
\end{equation}
\item[(2)] $B:V\times V\to V^*$ is a bilinear, continuous operator such that
\begin{equation}\label{property:B}
\langle B(u,v),v\rangle=0\ \mbox{for}\ v\in V.
\end{equation}
\end{itemize}

\begin{lemma} \label{prop-of-B}
Given $\lambda >0$ and $s\in\br$ there exists a smooth function $w\in H^1(\Omega;\br^2)$ such that
$\mathrm{div}\, w=0$ in $\Omega$, $w=0$ on $\Gamma_1$, $w$ is L-periodic in $x_1$, $w_T=(s,0)$, $w_N=0$ on $\Gamma_0$ and for all $v\in V$,
\begin{eqnarray} \label{estimb}
    |\langle B(v,w), v \rangle| \leq \lambda ||v||^2.
\end{eqnarray}
\end{lemma}
\begin{proof} Let $w$ be of the form $w(x_2) = (s\rho(x_2/h_0),0)$, where $\rho:[0,\infty)\to [0,1]$ is a smooth function such that $\rho(0)=1$, $\rho'(0)=0$,
$\mathrm{supp} \rho \subset [0,\min\{\frac{\lambda}{2|s|},h_0, 1)\}]$, and
$h_0$ is the minimum value of $h(x_1)$ on $[0,L]$. It is clear that all the stated properties of $w$ other then (\ref{estimb}) hold. To prove (\ref{estimb}) observe that under our assumptions
\begin{eqnarray} \label{ident1}
\int_\Omega|\mathrm{rot} w|^2dx=\int_\Omega|\nabla w|^2dx
\end{eqnarray}
 and then
\begin{eqnarray}
    |\langle B(v,w), v \rangle| &\leq& \|\nabla v\|_{L^2(\Omega)}
    \|x_2w(x_2)\|_{L^\infty([0,h_0])}\left\|\frac{v}{x_2}\right\|_{L^2(\Omega)}
                              \\            &\leq& ||v||\frac{\lambda}{2}2||v|| =
                              \lambda||v||^2\nonumber
\end{eqnarray}
in view of the Hardy inequality.
\end{proof}

\begin{lemma}\label{thm:a-priori}
Let $j:\br\to\br$ satisfy $H(j)$. For any solution $v$ of Problem \ref{problem_main} and for
a.e.~$t\in \br^+$,
\begin{eqnarray} \label{eqn:estimate_vprime}
    \|v'(t)\|_{V^\star} \leq C_3(1 + \|v(t)\| + \|v(t)\|_H^{1/2}\|v(t)\|^{\frac{3}{2}}),
\end{eqnarray}
with a constant $C_3> 0$ independent of $v$.
\end{lemma}
\begin{proof}
Let $v$ be a solution of Problem \ref{problem_main}. For any test function $z \in V$ we have, for a.e. $t\in \br^+$,
$$
|\langle v'(t),z \rangle | \leq \|F\|_{V^*}\|z\| + \|\xi(t)\|_{L^2(\Gamma_0)}\|\gamma\|\|z\| + |\langle G(v(t)) - Av(t) - B[v(t)], z\rangle |.
$$
From the growth condition $H(j)(b)$ it follows that $\|\xi(t)\|_{L^2(\Gamma_0)}\leq C(1+\|v(t)\|)$ with a constant $C>0$. Moreover $|\langle Av(t),z\rangle|+|\langle G(v(t)),z\rangle|\leq C\|v(t)\|\|z\|$. It remains to estimate the nonlinear term. For all $w\in V$ we have (\ref{ident1}).
Now, from H\"{o}lder's and Ladyzhenskaya's inequalities we obtain
\begin{eqnarray*}
\arrowvert\langle B[v(t)],z\rangle\arrowvert &\leq&
\int_\Omega|\mathrm{rot} v(t)||v(t)|| z|\, dx \leq
||v(t)||||v(t)||_{L^4(\Omega;\br^2)}||z||_{L^4(\Omega;\br^2)} \\ &\leq&
C||v(t)||_H^{1/2}||v(t)||^{3/2}||z||.
\end{eqnarray*}
In this way we obtain (\ref{eqn:estimate_vprime}).
\end{proof}

\begin{theorem} \label{thm-existence-sol}
Let the potential $j$ satisfy $H(j)$, $F\in V^\star$, and
$u_0\in H$. Then for every $v_0\in H$ there exists $v\in \mathcal{K}^+$ such that $v(0)=v_0$.
\end{theorem}
\begin{proof}
Let $\varrho\in C^\infty_0(-1,1)$ be a mollifier such that $\int_{-1}^1\varrho(s)\ ds=1$ and $\varrho(s)\geq 0$. We define
$\varrho_n:\br\to \br$ by $\varrho_n(s)=n\varrho(ns)$ for $n\in\mathbb{N}$ and $s\in \br$. Then $\mbox{supp}\ \varrho_n\subset \left(-\frac{1}{n},\frac{1}{n}\right)$.
We consider $j_n:\br\to \br$ defined by the convolution
$$
j_n(r)=\int_\br\varrho_n(s) j(r-s)\ ds\ \ \mbox{for}\ r\in\br.
$$
Note that $j_n\in C^\infty(\br)$. Moreover, by the computation analogous to proofs of Lemmas 5 and 9 in \cite{Kalita2010} it follows that for all $n\geq N_0$, where $N_0\in\mathbb{N}$ is given, regularized functions $j_n$ still satisfy $H(j)$, where the constants $c_1,c_2,d_1,d_2$ are different then the ones for $j$, but independent on $n$, and still we have $d_2\in \left(0,\frac{\nu}{\|\gamma\|^2}\right)$.

Let us furthermore take the sequence $V_n$ of finite dimensional spaces such that $V_n$ is spanned by the first $n$ eigenfuctions of the Stokes operator with the Dirichlet and periodic boundary conditions given in the definition of the space $V$. Then $\{V_n\}_{n=1}^\infty$ approximate $V$ from inside, i.e. $\overline{\bigcup_{n=1}^\infty V_n} = V$. Moreover we take the sequence $v_{0n}\to v_0$ strongly in $H$ such that $v_{0n}\in V_n$. We formulate
the regularized Galerkin problems for $n\in\mathbb{N}$:\\
Find a continuous function $v_n:\br^+\to V_n$ such that for a.e. $t\in \br^+$ $v_n$ is differentiable and
\begin{eqnarray} \label{def-sol-3}
 &&\langle v_n'(t) + Av_n(t) + B[v_n(t)],z \rangle + (j_n'((v_n)_N(t)(\cdot)),z_N)_{L^2(\Gamma_0)} =\\
     &&= \langle F + G(v_n(t)), z \rangle,\nonumber\\
     && v_n(0)=v_{0n} \label{ini-3}
\end{eqnarray}
for a.e. $t\in \br^+$ and for all $z\in V_n$.

We first show that if $v_n$ solves (\ref{def-sol-3}) then an estimate analogous to (\ref{eqn:estimate_v}) holds.

We take $z=v_n(t)$ in (\ref{def-sol-3}) and, using (\ref{property:A}) and (\ref{property:B}), as well as the fact that $\langle B(v,w),v\rangle \leq \lambda \|v\|^2$ for all $v\in V$ where $\lambda$ can be made arbitrarily small (Lemma \ref{prop-of-B}), we obtain
$$
\frac{1}{2}\frac{d}{dt}\|v_n(t)\|_H^2+\nu\|v_n(t)\|^2+ (j_n'((v_n)_N(t)(\cdot)),(v_n)_N(t))_{L^2(\Gamma_0)}
     \leq \|F\|_{V^*}\|v_n(t)\| +  \lambda\|v_n(t)\|^2,
$$
for a.e. $t\in \br^+$. Using $H(j)(c)$ and the Cauchy inequality with some $\varepsilon>0$ we obtain
$$
\frac{1}{2}\frac{d}{dt}\|v_n(t)\|_H^2+\nu\|v_n(t)\|^2+ d_1m(\Gamma_0) - d_2\|(v_n)_N(t)\|^2_{L^2(\Gamma_0)}
     \leq C(\varepsilon)\|F\|^2_{V^*} +  (\lambda+\varepsilon)\|v_n(t)\|^2,
$$
for a.e. $t\in \br^+$, where $\varepsilon>0$ is arbitrary and the constant $C(\varepsilon)>0$. Note that by the trace theorem
$\|(v_n)_N(t)\|^2_{L^2(\Gamma_0)}\leq \|v_n(t)\|^2_{L^2(\Gamma_0;\br^2)}\leq \|\gamma\|^2\|v_n(t)\|^2$. It follows that
$$
\frac{1}{2}\frac{d}{dt}\|v_n(t)\|_H^2+(\nu-d_2\|\gamma\|^2)\|v_n(t)\|^2
     \leq C(\varepsilon)\|F\|^2_{V^*} +  (\lambda+\varepsilon)\|v_n(t)\|^2+ |d_1|m(\Gamma_0),
$$
for a.e. $t\in \br^+$.
It is enough to take $\lambda=\varepsilon=\frac{\nu-d_2\|\gamma\|^2}{4}$. We get, with $C_1,C_2>0$ independent of $t$,
\begin{equation}\label{eqn:galerkin_estimate_v}
\frac{1}{2}\frac{d}{dt}\|v_n(t)\|_H^2+C_1\|v_n(t)\|^2
     \leq C_2(1+\|F\|^2_{V^*}).
\end{equation}
Note that, after integration, we have for all $t\geq 0$
\begin{equation}\label{eqn:integral}
\frac{1}{2}\|v_n(t)\|_H^2+C_1\int_{0}^t\|v_n(s)\|^2\ ds\leq \frac{1}{2}\|v_{0n}\|_H^2 + C_2(1+\|F\|^2_{V^*})t
\end{equation}
Existence of the solution to the Galerkin problem (\ref{def-sol-3}) is standard and follows by the Caratheodory theorem and estimate (\ref{eqn:integral}). Note that for all $n\in N$ solutions of (\ref{def-sol-3}) satisfy the estimate
of Lemma \ref{thm:a-priori} where the constants do not depend on initial conditions and the dimension $n$.
We deduce from (\ref{eqn:galerkin_estimate_v})  that for all $T>0$
\begin{equation}\label{eqn:v_n-bounded}
v_n\ \ \mbox{is bounded}\ \ \mbox{in}\ \ L^2(0,T;V)\cap L^\infty(0,T;H).
\end{equation}
From Lemma \ref{thm:a-priori} it follows that for all $T>0$
\begin{equation}\label{eqn:v_n-prime-bounded}
v_n'\ \ \mbox{is bounded in}\ \ L^\frac{4}{3}(0,T;V^*).
\end{equation}
In view of (\ref{eqn:v_n-bounded}) and (\ref{eqn:v_n-prime-bounded}), by diagonalization, we can construct a subsequence such that
\begin{eqnarray}\label{eqn:limit_vn}
v_n \to v\ \ \mbox{weakly in}\ \ L^2_{loc}(\br^+;V),
\end{eqnarray}
and
\begin{eqnarray}\label{eqn:limit_vnprime}
v_n' \to v'\ \ \mbox{weakly in}\ \ L^\frac{4}{3}_{loc}(\br^+;V^*).
\end{eqnarray}

First we show that $v(0)=v_0$ in $H$. To this end choose $T>0$. We have for $t\in (0,T)$
$$
v_n(t)=v_{n0} + \int_0^t v_n'(s)\, ds, \quad v(t)=v(0) + \int_0^tv'(s)\,ds,
$$
where the equalities hold in $V^*$, and take $\phi\in C([0,T];V)$ with $\int_0^T\phi(t)dt\not=0$. Then
$$
\int_0^T\langle v_n(t)-v(t),\phi(t)\rangle\, dt = \int_0^T\langle v_{n0}-v(0),\phi(t)\rangle \,dt +
\int_0^T\left\langle\int_0^t(v_n'(s)-v'(s))\, ds, \phi(t)\right \rangle\, dt.
$$
The left-hand side of this equality goes to zero with $n\to\infty$ as $v_n\to v$ weakly in $L^2(0,T;V)$.
We shall prove that the last integral on the right-hand side also goes to zero with
$n\to\infty$ so that $v(0)=v_0$ in $H$. We have,
$$
f_n(t)= \left\langle\int_0^t(v_n'(s)-v'(s))ds, \phi(t)\right\rangle
=  \int_0^t\langle v_n'(s)-v'(s), \phi(t)\rangle ds,
$$
where $f_n(t)\to 0$ as $n\to\infty$ as $v_n'\to v'$ weakly in $L^\frac{4}{3}(0,T;V^*)$, and
$$
|f_n(t)| \leq ||v_n'-v'||_{L^\frac{4}{3}(0,T;V^*)}T^\frac{1}{4}\max_{[0,T]}\|\phi(t)\|
\leq C T^\frac{1}{4}\max_{[0,T]}\|\phi(t)\|.
$$
By the Lebesque dominated convergence lemma the result follows.

Now we pass with $n\to \infty$ in the equation (\ref{def-sol-3}). To this end let us choose $T>0$. We multiply (\ref{def-sol-3}) by $\phi\in C([0,T])$
and integrate with respect to t in $[0,T]$. Passing to the limit in linear terms $A$ and $G$ is standard. We focus on the multivalued term and the convective term.

Let $Z=H^{1-\delta}(\Omega;\mathbb{R}^2)$ for $\delta\in \left(0,\frac{1}{2}\right)$.
Since $V\subset Z$ compactly, from the Lions-Aubin lemma it follows that, for a subsequence, $v_n\to v$ strongly in $L^2(0,T;Z)$, and, by continuity of the trace operator,
\begin{equation}\label{eqn:normal-trace-strong}
(v_n)_N\to v_N\ \ \mbox{strongly in}\ \ L^2(0,T;L^2(\Gamma_0)).
\end{equation}

From the growth condition $H(j)(b)$ valid for all $j_n'$ with the constants independent of $n\in\mathbb{N}$ and from (\ref{eqn:normal-trace-strong}) it follows that
the sequence $j_n'((v_n)_N(\cdot,\cdot))$ is bounded in $L^2(0,T;L^2(\Gamma_0))$, and we can
extract a subsequence (renamed  $j_n'((v_n)_N(\cdot,\cdot))$) such that
\begin{equation}\label{eqn:derivative-weak}
j_n'((v_n)_N(\cdot,\cdot)) \to \xi\ \ \mbox{weakly in}\ \ L^2(0,T;L^2(\Gamma_0)).
\end{equation}

We need to show that $\xi\in S^2_{\partial j(v_N(\cdot,t))}$. The proof of this fact follows the lines of the proof of Step III of Theorem 1 in \cite{Miettinen1999}. Let us denote $\Gamma_{0T}=\Gamma_0\times (0,T)$. From (\ref{eqn:normal-trace-strong}) it follows that, for a subsequence, $(v_n)_N(x,t)\to v_N(x,t)$ for a.e. $(x,t)\in \Gamma_{0T}$ and $|(v_n)_N(x,t)|\leq h(x,t)$ for some $h\in L^2(\Gamma_{0T})$. Moreover from (\ref{eqn:derivative-weak}) it follows that $j_n'((v_n)_N(\cdot,\cdot))\to \xi(\cdot,\cdot)$ weakly in $L^1(\Gamma_{0T})$.
Now take $w\in L^\infty(\Gamma_{0T})$. We have
 $$
 \int_{\Gamma_{0T}}\xi(x,t)w(x,t)\ d\sigma(x)\,dt=\lim_{n\to\infty}\int_{\Gamma_{0T}}j_n'((v_n)_N(x,t))w(x,t)\ d\sigma(x)\,dt.
 $$
From $H(j)(b)$ we can invoke the Fatou lemma and obtain
 \begin{eqnarray}
 &&\int\limits_{\Gamma_{0T}}\xi(x,t)w(x,t)\ d\sigma(x)\,dt\leq \int_{\Gamma_{0T}}\limsup_{\substack{n\to\infty,\\ \lambda\to 0^+}}\frac{j_n((v_n)_N(x,t)+\lambda w(x,t))-j_n((v_n)_N(x,t))}{\lambda}\ d\sigma(x)\,dt=\nonumber\\
 && = \int\limits_{\Gamma_{0T}}\limsup_{\substack{n\to\infty,\\ \lambda\to 0^+}}\int\limits_{\br}\varrho_n(\tau)\frac{j((v_n)_N(x,t)-\tau+\lambda w(x,t))-j((v_n)_N(x,t)-\tau)}{\lambda}\ d\tau\,d\sigma(x)\,dt\leq\nonumber \\
 &&\leq \int\limits_{\Gamma_{0T}}\limsup_{\substack{n\to\infty,\\ \lambda\to 0^+,\\z(x,t)\to 0}}\int\limits_{\br}\varrho_n(\tau)\frac{j(v_N(x,t)+z(x,t)+\lambda w(x,t))-j(v_N(x,t)+z(x,t))}{\lambda}\ d\tau\,d\sigma(x)\,dt=\nonumber\\
 &&=\int\limits_{\Gamma_{0T}}j^0(v_N(x,t);w(x,t))\ \,d\sigma(x)\,dt
 \end{eqnarray}
 Since the choice of $w$ is arbitrary, from the definition of the generalized gradient we get
 $$
 \xi(x,t)\in \partial j(v_N(x,t))\ \ \mbox{a.e.}\ \ (x,t)\in\Gamma_{0T},
 $$
 and the desired result follows.

Now we show the convergence in the nonlinear term, namely, that (for a subsequence)
\begin{equation} \label{nonlinear-new2}
\int_0^T\int_\Omega(\mathrm{rot}\,v_n(t)\times v_n(t))\cdot z(x)\phi(t)dxdt \to
\int_0^T\int_\Omega(\mathrm{rot}\,v(t)\times v(t))\cdot z(x)\phi(t)dxdt.
\end{equation}
First we prove that
\begin{eqnarray}\label{newformula1}
&&\int_0^T\int_\Omega(\mathrm{rot}\,v_n(t)\times v_n(t))\cdot z(x)\phi(t)dxdt
= \frac{1}{2}\int_0^T\int_{\Gamma_0}(v_{n})_N^2(x,t)z_N(x)\phi(t)d\sigma(x)dt \nonumber \\
&-&\int_0^T\int_\Omega(v_n(x,t)\cdot\nabla)z(x)v_n(x,t)\phi(t)dxdt.
\end{eqnarray}
From (\ref{general-formula}),  as well as the formulas
$(a\times b)\cdot c=(b\times c)\cdot a$, and
\begin{eqnarray} \label{form2}
\nabla\cdot F=0, \nabla\cdot G=0 \Longrightarrow \mathrm{rot}\,(F\times G)= (G\cdot\nabla)F - (F\cdot\nabla)G,
\end{eqnarray}
we have
\begin{eqnarray}
\langle B[v_n(t)],z\rangle = \int_\Omega(v_n(t)\times z)\cdot\mathrm{rot}\,v_n(t)\,dx&=&\\
\int_\Omega\mathrm{rot}(v_n(t)\times z)\cdot v_n(t)\,dx &+&
\int_{\partial\Omega}(v_n(t)\times(v_n(t)\times z))\cdot n\,d\sigma(x).\nonumber
\end{eqnarray}
The surface integral is equal to zero and hence, using (\ref{form2}) in the right hand side, we obtain
\begin{equation} \label{our-formula}
\langle B[v_n(t)],z\rangle = \int_\Omega(z\cdot\nabla)v_n(t)\cdot v_n(t)dx
 -
\int_\Omega(v_n(t)\cdot\nabla)z\cdot v_n(t) dx.
\end{equation}
Intergation by parts in the first integral on the right hand side and then integration in the time variable give the result.

Since we can deal with the second term on the right-hand side of (\ref{newformula1}) as in the usual Navier-Stokes theory
\cite{temam-navier}, we consider only the surface integral. Taking the difference of the corresponding terms and setting $z_N(x)\phi(t) = \psi(x,t)$ we obtain
\begin{eqnarray*}
&&|\int_0^T\int_{\Gamma_0}((v_{n})_N(x,t)-v_N(x,t))((v_{n})_N(x,t)+v_N(x,t))\psi(x,t)\, d\sigma(x)dt| \\
&\leq&
||(v_n)_N-v_N||_{L^2(0,T;L^2(\Gamma_0))}
||v_n+v||_{L^2(0,T;L^4(\Gamma_0;\br^2))}\max_{t\in[0,T]}||\psi(t)||_{L^4(\Gamma_0)}.
\end{eqnarray*}
This proves (\ref{nonlinear-new2}) as $v_n$ is bounded in
$L^2(0,T;L^4(\Gamma_0;\br^2))$ ($H^{1/2}(\Gamma_0;\br^2)\subset L^r(\Gamma_0;\br^2)$ continuously for
every $r\geq 1$) and $\|(v_n)_N-v_N\|_{L^2(0,T;L^2(\Gamma_0))}\to 0$ by (\ref{eqn:normal-trace-strong}).

Hence, the limit $v$ solves Problem \ref{problem_main}. It remains to show (\ref{eq:estimate_v_meaning}). The proof follows the lines of that of Theorem 8.1 in \cite{chep-vish-1997}. Let us fix $0\leq  t_1 < t_2$ and choose $\psi\in C^\infty_0(t_1,t_2)$ with $\psi\geq 0$. We multiply (\ref{eqn:galerkin_estimate_v}) by $\psi$ and integrate by parts. We have
\begin{equation}\label{eqn:aprioriintegral}
-\frac{1}{2}\int_{t_1}^{t_2}\|v_n(t)\|_H^2\psi'(t)\ dt+C_1\int_{t_1}^{t_2}\|v_n(t)\|^2\psi(t)\ dt
     \leq C_2(1+\|F\|^2_{V^*})\int_{t_1}^{t_2}\psi(t)\ dt.
\end{equation}
We need to pass to the limit ($n\to \infty$) in the last inequality. From (\ref{eqn:limit_vn}) and (\ref{eqn:limit_vnprime}) by the Lions-Aubin lemma we conclude that $v_n\to v$ strongly in $L^2(t_1,t_2;H)$. From inequality
$$
\int_{t_1}^{t_2}(\|v_n(t)\|_H-\|v(t)\|_H)^2\ dt\leq \int_{t_1}^{t_2}\|v_n(t)-v(t)\|_H^2\ ds
$$
it follows that $\|v_n(t)\|_H\to \|v(t)\|_H$ strongly in $L^2(t_1,t_2)$, and hence, for a subsequence, $\|v_n(t)\|_H^2\to \|v(t)\|_H^2$ for almost every $t\in (t_1,t_2)$. Since from (\ref{eqn:integral}) it follows that functions $\|v_n(t)\|_H^2\psi'(t)$ have an integrable majorant on $(t_1,t_2)$, we have, by Lebesgue dominated convergence theorem that
\begin{equation}
\int_{t_1}^{t_2}\|v_n(t)\|_H^2\psi'(t)\ dt\to \int_{t_1}^{t_2}\|v(t)\|_H^2\psi'(t)\ dt.
\end{equation}
Now from (\ref{eqn:limit_vn}) it follows that $v_n(t)(\psi(t))^{\frac{1}{2}}\to v(t)(\psi(t))^{\frac{1}{2}}$ weakly in $L^2(t_1,t_2;V)$ and hence by weak lower semicontinuity of the norm we obtain
\begin{equation}
\int_{t_1}^{t_2}\|v(t)\|^2\psi(t)\ dt\leq \liminf_{n\to\infty} \int_{t_1}^{t_2}\|v_n(t)\|_H^2\psi(t)\ dt.
\end{equation}
Thereby, we can pass to the limit in (\ref{eqn:aprioriintegral}) which gives us (\ref{eq:estimate_v_meaning}) and the proof is complete.
\end{proof}

\renewcommand{\theequation}{\arabic{section}.\arabic{equation}}
\setcounter{equation}{0}
\section{Existence of attractors} \label{trajectory-attractor}
In this section we prove the existence of a trajectory attractor and a weak global attractor for the considered problem.

We have shown that for any initial condition $v_0\in H$ Problem \ref{problem_main} has at least one solution
$v\in\C{K}_+$. The key idea behind the trajectory attractor (see
\cite{chep-vish-1997}, \cite{chep-vish-2002}, \cite{chep-vish-2002-paper}, \cite{chep-vish-2011}) is, in contrast to the direct study of
the solutions asymptotic behavior, the investigation of the family of shift operators $\{T(t)\}_{t\geq 0}$ defined on $\C{K}^+$ by the formula
$$
(T(t)v)(s) = v(s+t).
$$
Before we pass to a theorem on existence of the trajectory attractor, which is the main result of this section, we recall some definitions and results
(see \cite{chep-vish-1997}, \cite{chep-vish-2002}, \cite{chep-vish-2002-paper}, \cite{chep-vish-2011}).

\noindent Let
$$\C{F}(0,T)=\{u\in L^2(0,T;V)\cap L^\infty(0,T;H): u'\in L^\frac{4}{3}(0,T;V^*)\}$$ and
$$\C{F}^{loc}_+=\{u\in L^2_{loc}(\mathbb{R}^+;V)\cap L^\infty_{loc}(\mathbb{R}^+;H) : u'\in L^\frac{4}{3}_{loc}(\mathbb{R}^+;V^*)\}.$$

\noindent Moreover we define $C_w([0,T];H)$ as the space of all functions $u:[0,T]\to H$ such that for any $\phi\in H$ the scalar product $(u(t),\phi)_H$ is continuous for $t\in [0,T]$. Note that $\C{F}(0,T)\subset C_w([0,T];H)$ and from the Lions-Magenes lemma (c.f., Lemma 1.4, Chapter~3 in \cite{temam-navier}, Theorem II,1.7 in \cite{chep-vish-2002} or Lemma 2.1 in \cite{chep-vish-2011}) it follows that for all $u\in \C{F}(0,T)$ we have
\begin{equation}\label{eq:lions-magenes}
\|u(t)\|_H\leq \|u\|_{L^\infty(0,T;H)}\quad\mbox{for all}\quad t \in [0,T].
\end{equation}
Furthermore, let us define the Banach space
$$
\C{F}^b_+ = \{u\in \C{F}^{loc}_+: \|u\|_{\C{F}^b_+}<\infty\},
$$
where the norm in $\C{F}^b_+$ is given by
$$
\|u\|_{\C{F}^b_+} = \sup_{h\geq 0}\|\Pi_{0,1} u(\cdot + h)\|_{\C{F}(0,1)}.
$$
In the above definition $\Pi_{0,1}u(\cdot)$ is a restriction of $u(\cdot)$ to the interval $(0,1)$, and
$$
||v||_{\C{F}(0,1)} =
\|v\|_{L^2(0,1;V)}+||v||_{L^\infty(0,1;H)} + \|v'\|_{L^{\frac{4}{3}}}(0,1;V^*).
$$
Finally, we define the topology by $\Theta^{loc}_+$ in the space $\C{F}^{loc}_+$ in the following way: the sequence $\{u_n\}_{n=1}^\infty\subset \C{F}^{loc}_+$ is said to converge to $u\in \C{F}^{loc}_+$ in the sense of $\Theta^{loc}_+$ we have
\begin{eqnarray}
&& u_n\to u\ \ \mbox{weakly in}\ \ L^2_{loc}(\br^+;V)\nonumber\\
&& u_n\to u\ \ \mbox{weakly-$*$ in}\ \ L^\infty_{loc}(\br^+;H)\nonumber\\
&& u_n'\to u'\ \ \mbox{weakly in}\ \ L^\frac{4}{3}_{loc}(\br^+;V^*)\nonumber.
\end{eqnarray}
Note (see for example \cite{chep-vish-2011}) that the topology $\Theta^{loc}_+$ is stronger then the topology of $C_w([0,T];H)$ for all $T\geq 0$ and hence from the fact that $v_n\to v$ in $\Theta^{loc}_+$ it follows that $v_n(t)\to v(t)$ weakly in $H$ for all $t\geq 0$.
\begin{definition}
A set $\C{P}\subset \C{K}^+$ is said to be \textit{absorbing} for the shift semigroup $\{T(t)\}$ if for any set $\mathfrak{B}\subset \C{K}^+$
bounded in $\C{F}^b_+$ there exists $\tau = \tau(\mathfrak{B})>0$ such that $T(t)\mathfrak{B} \subset\C{P}$ for all $t\geq \tau$
\end{definition}
\begin{definition}
A set $\C{P}\subset \C{K}^+$ is said to be \textit{attracting} for the shift semigroup $\{T(t)\}$ in the topology $\Theta^{loc}_+$ if for any set $\mathfrak{B}\subset \C{K}^+$ bounded in $\C{F}^b_+$ and for any neighborhood $\C{O}(\C{P})$ of $\C{P}$ in the topology $\Theta^{loc}_+$ there exists $\tau = \tau(\mathfrak{B},\C{O})>0$ such that $T(t)\mathfrak{B} \subset\C{O}(\C{P})$ for all $t\geq \tau$.
\end{definition}
\begin{definition}
A set $\mathfrak{U}\subset \C{K}^+$ is called a trajectory attractor of the shift semigroup $\{T(t)\}$ on $\C{K}^+$ in the topology
$\Theta^{loc}_+$ if
\begin{itemize}
\item[(a)] $\mathfrak{U}$ is bounded in $\C{F}^b_+$ and compact in the topology $\Theta^{loc}_+$,
\item[(b)] $\mathfrak{U}$ is an attracting set in the topology $\Theta^{loc}_+$,
\item[(c)] $\mathfrak{U}$ is strictly invariant, i.e. $T(t)\mathfrak{U}=\mathfrak{U}$ for any $t\geq 0$.
\end{itemize}
\end{definition}
In order to show the existence of a trajectory attractor for Problem \ref{problem_main} we will use the following theorem (see for instance Theorem 4.1 in \cite{chep-vish-2011}).
\begin{theorem}\label{theorem-vishik}
Assume that the trajectory set $\C{K}^+$ is contained in the space $\C{F}^b_+$ and that
\begin{equation}\label{eqn:semiinvariance}
T(t)\mathcal{K}^+\subset \mathcal{K}^+.
\end{equation}
Suppose that the semigroup $\{T(t)\}$ has an attracting set $\C{P}$ that is bounded in the norm of  $\C{F}^b_+$ and compact in
the topology $\Theta^{loc}_+$. Then the shift semigroup $\{T(t)\}$ has a trajectory attractor $\mathfrak{U}\subset \C{P}$.
\end{theorem}
We are in position to formulate the main theorem of this section.
\begin{theorem}\label{thm:w_trajectory_attractor}
The shift semigroup $\{T(t)\}_{t\geq 0}$ defined on the set $\C{K}^+$ of solutions of  Problem \ref{problem_main} has a trajectory attractor
$\mathfrak{U}$ which is bounded in $\C{F}^b_+$ and compact in the topology $\Theta^{loc}_+$.
\end{theorem}
Before we pass to the proof of this theorem we will need two auxiliary lemmas.
\begin{lemma}\label{lemma:estimates_b}
Let $v\in \C{K}^+$. For all $h\geq 0$ the following estimates hold
\begin{eqnarray}
&&\|v\|^2_{L^2(h,h+1;V)} + \|v\|_{L^\infty(h,h+1;H)}^2\leq C_4 +  C_5 \|v\|_{L^\infty(0,1;H)}^2 e^{-C_6 h},\label{eqn:estimate_1}\\
&&\|v'\|^{\frac{4}{3}}_{L^{\frac{4}{3}}(h,h+1;V^*)}  \leq C_7 +  C_8 \|v\|_{L^\infty(0,1;H)}^\frac{8}{3} e^{-C_9 h},\label{eqn:estimate_2}
\end{eqnarray}
where $C_4, C_5, C_6, C_7, C_8$, and $C_9$ are positive constants independent of $h, v$.
\end{lemma}
\begin{proof}
Let us fix $h\geq 0$ and $v\in \C{K}^+$. We observe that, since for all $u\in V$ we have $\lambda_1 \|u\|^2_H\leq \|u\|^2$ with the constant $\lambda_1>0$, it follows from (\ref{eq:estimate_v_meaning}), that for all $0\leq t_1<t_2$ and $\psi\in C^\infty_0(t_1,t_2)$ with $\psi\geq 0$,
\begin{equation}\label{eq:estimate_v_meaning_h}
-\frac{1}{2}\int_{t_1}^{t_2}\|v(t)\|_H^2\psi'(t)\ dt + C_1\lambda_1\int_{t_1}^{t_2}\|v(t)\|_H^2\psi(t)\ dt\leq C_2(1+\|F\|^2_{V^*})\int_{t_1}^{t_2}\psi(t)\ dt.
\end{equation}
We are in position to use Lemma 9.2 from \cite{chep-vish-1997} to deduce that there exists a set $Q\subset \br^+$ of zero measure such that for all
$t,\tau \in \br^+\setminus Q$ we have
\begin{equation}\label{eqn:from_gronwall}
\|v(t)\|_H^2e^{C_1\lambda_1 (t-\tau)} - \|v(\tau)\|_H^2 \leq \frac{e^{C_1\lambda_1(t-\tau)}-1}{C_1\lambda_1}C_2(1+\|F\|^2_{V^*}).
\end{equation}
Now since the function $t\to v(t)$ is weakly continuous, it follows that $t\to \|v(t)\|_H^2$ is lower semicontinuous and (\ref{eqn:from_gronwall}) holds for all $t\geq \tau$. Hence, for all $t > 0$ we can find $\tau\in (0,\min\{1,t\})$ such that
\begin{equation}\label{eqn:from_gronwall_1}
\|v(t)\|_H^2e^{C_1\lambda_1 (t-\tau)}  \leq \|v\|_{L^\infty(0,1;H)}^2+\frac{e^{C_1\lambda_1(t-\tau)}}{C_1\lambda_1}C_2(1+\|F\|^2_{V^*}),
\end{equation}
and moreover, for all $t > 0$, we have
\begin{equation}\label{eqn:from_gronwall_2}
\|v(t)\|_H^2  \leq e^{C_1\lambda_1 (1-t)} \|v\|_{L^\infty(0,1;H)}^2+\frac{C_2(1+\|F\|^2_{V^*})}{C_1\lambda_1}.
\end{equation}
Now by Corollary 9.2 in \cite{chep-vish-1997} it follows from (\ref{eq:estimate_v_meaning}) that for almost all $h > 0$ we have
$$
\frac{1}{2}\|v(h+1)\|_H^2+C_1\int^{h+1}_h\|v(t)\|^2\ dt \leq C_2(1+\|F\|_{V^*}^2)+\frac{1}{2}\|v(h)\|_H^2.
$$
Using (\ref{eqn:from_gronwall_2}) we obtain, for a.e. $h > 0$,
$$
\int^{h+1}_h\|v(t)\|^2\ dt \leq \frac{C_2}{C_1}(1+\|F\|_{V^*}^2)+\frac{C_2(1+\|F\|^2_{V^*})}{2C_1^2\lambda_1} +  \frac{\|v\|_{L^\infty(0,1;H)}^2 e^{C_1\lambda_1(1-h)} }{2 C_1},
$$
and the estimate (\ref{eqn:estimate_1}) follows for all $h\geq 0$ since both left- and right-hand side of the above inequality are continuous functions of $h$.
Now using
(\ref{eqn:estimate_vprime}) we get for a.e. $t\in \br^+$
$$
\|v'(t)\|_{V^*}^{\frac{4}{3}}\leq C_{10}\left(1+\left(1+\|v(t)\|_H^{\frac{2}{3}}\right)\|v(t)\|^2\right),
$$
where $C_{10}>0$. Integrating this inequality from $h$ to $h+1$ we obtain
$$
\int_{h}^{h+1}\|v'(t)\|_{V^*}^{4/3}\ dt\leq
C_{10} + C_{10}\left(1+||v(t)||_{L^\infty(h,h+1;H)}^{\frac{2}{3}}\right)\int_{h}^{h+1}\|v(t)\|^2\ dt.
$$
Inequality (\ref{eqn:estimate_2}) follows from an application of (\ref{eqn:estimate_1}).
\end{proof}
\begin{lemma} \label{lemma-usc}The multifunction $S^2_{\partial j(\cdot)}:L^2(\Gamma_0)\to 2^{L^2(\Gamma_0)}$ is strong-weak upper semicontinuous and has nonempty, bounded, closed, convex and hence weakly compact values.
\end{lemma}
\begin{proof}
First observe that from the growth condition $H(j)(b)$ it follows that if $\xi(x)\in \partial j(u(x))\ \mbox{for a.e.}\ x\in \Gamma_0$ for $u\in L^2(\Gamma_0)$,
then $\xi \in L^2(\Gamma_0)$. Hence, and from the fact that $\partial j(s)$ is nonempty for all $s\in \br$ it follows that $S^2_{\partial j(u(\cdot))}$ is nonempty for all $u\in L^2(\Gamma_0)$. From $H(j)(b)$ it also follows that this set is bounded, since $\|\xi\|^2_{L^2(\Gamma_0)}\leq 2c_1^2\sigma(\Gamma_0)+2c_2^2\|u\|_{L^2(\Gamma_0)}$. Moreover from the fact that $\partial j(s)$ is convex for all $s\in \br$ it follows that $S^2_{\partial j(u(\cdot))}$ is convex for all $u\in L^2(\Gamma_0)$. We will demonstrate that $S^2_{\partial j(u(\cdot))}$ is a closed set for all $u\in L^2(\Gamma_0)$. Let $\xi_n\in S^2_{\partial j(u(\cdot))}$ and $\xi_n\to \xi$ strongly in $L^2(\Gamma_0)$. Hence, for a subsequence $\xi_n(x)\to \xi(x)$ for a.e. $x\in \Gamma_0$. Since $\partial j$ is upper semicontinuous multifunction it follows that $\xi(x)\in \partial j(u(x))$ for a.e. $x\in \Gamma_0$ and hence $\xi \in S^2_{\partial j(u(\cdot))}$. We have shown that $S^2_{\partial j(\cdot)}$ has nonempty and weakly compact values. In order to finish the proof of upper semicontinuity, in view of Proposition 4.1.11 in \cite{DMP1} it is enough to show that if $u_n\to u$ strongly in $L^2(\Gamma_0)$ and $\xi_n\in S^2_{\partial j(u_n(\cdot))}$ then, for a subsequence, $\xi_n\to \xi$ weakly in $L^2(\Gamma_0)$ with $\xi\in S^2_{\partial j(u(\cdot))}$. Let us choose $u_n\to u$ strongly in $L^2(\Gamma_0)$ and $\xi_n\in S^2_{\partial j(u_n(\cdot))}$.
From the growth condition $H(b)(ii)$ if follows that for a subsequence, $\xi_n\to \xi$ weakly in $L^2(\Gamma_0)$. The argument from now on
will follow the lines of the proof of Theorem 4.1 in \cite{naniewicz-2004}. For a subsequence we have $u_n(x)\to u(x)$ for a.e. $x\in \Gamma_0$ with $|u_n(x)|\leq h(x)$ for a.e. $x\in \Gamma_0$ with $h\in L^2(\Gamma_0)$. For any $z\in L^\infty(\Gamma_0)$, by the definition of the Clarke subdifferential, there holds
\begin{equation}\label{temp_eq_1}
\int_{\Gamma_0}\xi_n(x) z(x)d\sigma(x)
\leq\int_{\Gamma_0}j^{0}(u_n(x);z(x))d\sigma(x).
\end{equation}
Passing to the limit in  (\ref{temp_eq_1}) and using the weak
convergence of $\xi_n$ we have
\begin{equation}\label{temp_eq_2}
\int_{\Gamma_0}\xi(x)z(x)d\sigma(x)=\lim_{n\to\infty}\int_{\Gamma_0}\xi_n(x)
z(x)d\sigma(x)
\leq\limsup_{n\to\infty}\int_{\Gamma_0}j^{0}(u_n(x);z(x))d\sigma(x).
\end{equation}
Moreover, we have for a.e. $x\in \Gamma_0$
\begin{align}\label{temp_eq_3}
&j^0(u_n(x);z(x))=\sup_{\eta\in\partial j(u_n(x))}\eta
z(x)\leq\|z\|_{L^{\infty}(\Gamma_0)}\sup_{\eta\in\partial
j(u_n(x))}|\eta|\leq\nonumber\\
&\|z\|_{L^{\infty}(\Gamma_0)}(c_1+c_2|u_n(x)|)\leq \|z\|_{L^{\infty}(\Gamma_0)}(c_1+c_2h(x)),
\end{align}
where we used the growth condition $H(j)(b)$. We are in position to invoke the Fatou lemma in
(\ref{temp_eq_2}) which gives
\begin{equation}\label{temp_eq_4}
\int_{\Gamma_0}\xi(x)\cdot z(x)d\sigma(x)
\leq\int_{\Gamma_0}\limsup_{n\to\infty}j^{0}(u_n(x);z(x))d\sigma(x).
\end{equation}
From (\ref{temp_eq_4}) and the upper semicontinuity of the Clarke
directional derivative, we obtain
\begin{equation}\label{temp_eq_5}
\int_{\Gamma_0}\xi(x)\cdot z(x)d\sigma(x)
\leq\int_{\Gamma_0}j^{0}(u(x);z(x))d\sigma(x).
\end{equation}
Since in (\ref{temp_eq_5}) $z$ is arbitrary, it easily follows that
\begin{equation}\label{temp_eq_6}
\xi(x)\in\partial j(u(x))
\,\,\,\,a.e.\,\,in\,\,\Gamma_0,
\end{equation}
and the assertion if proved.\end{proof}

\begin{lemma}\label{lemma:closed}
The set $\C{K}^+$ is closed in the topology $\Theta^{loc}_+$.
\end{lemma}
\begin{proof}
Assume that for a sequence $\{v_n\}_{n=1}^\infty\subset \C{K}^+$ we have
\begin{eqnarray}
&& v_n\to v\ \ \mbox{weakly in}\ \ L^2_{loc}(\br^+;V)\nonumber\\
&& v_n\to v\ \ \mbox{weakly-$*$ in}\ \ L^\infty_{loc}(\br^+;H)\nonumber\\
&& v_n'\to v'\ \ \mbox{weakly in}\ \ L^\frac{4}{3}_{loc}(\br^+;V^*)\nonumber.
\end{eqnarray}
We need to show that $v$ satisfies (\ref{def-sol}) and (\ref{eq:estimate_v_meaning}). Since $\{v_n\}_{n=1}^\infty\subset \C{K}^+$, we have, for all $n\in\mathbb{N}$ and $z\in V$,
\begin{eqnarray*}
   &&  \langle v_n'(t) + Av_n(t) + B[v_n(t)],z \rangle + (\xi_n(t),z_N)_{L^2(\Gamma_0)}
     = \langle F, z \rangle +  \langle G(v_n(t)), z \rangle,\\
   &&  \xi_n(t)\in S^2_{\partial j((v_n)_N(\cdot,t))}, \nonumber
\end{eqnarray*}
a.e. $t\in \mathbb{R}^+$. Passing to the limit in terms with $A, B, G$ is analogous to that in the proof of Theorem \ref{thm-existence-sol} (see also \cite{chep-vish-2002}, \cite{temam-navier}, \cite{temam-infty}, \cite{Migorski-Ochal2007}).

In order to pass to the limit in the multivalued term observe that from $H(j)(b)$ it follows that for a.e. $t\in \br^+$
$$
\|\xi_n(t)\|^2_{L^2(\Gamma_0)}\leq 2c_1^2m(\Gamma_0)+2c_2\|\gamma\|^2\|v_n(t)\|^2.
$$
Hence, after integration in the time variable we get, for all $T\in \br^+$,
$$
\|\xi_n\|^2_{L^2(0,T;L^2(\Gamma_0))}\leq 2Tc_1^2m(\Gamma_0)+2c_2\|\gamma\|^2\|v_n\|^2_{L^2(0,T;V)}.
$$
It follows that there exists $\xi\in L^2_{loc}(\br^+;L^2(\Gamma_0))$ such that, for a subsequence,
$$
\xi_n\to \xi\ \ \mbox{weakly in}\ \ L^2_{loc}(\br^+;L^2(\Gamma_0)).
$$
As in the proof of Theorem \ref{thm-existence-sol}, we have
$$
(v_n)_N\to v_N\ \ \mbox{strongly in}\ \ L^2_{loc}(\br^+;L^2(\Gamma_0)),
$$
and, for a subsequence,
$$
(v_n)_N(t)\to v_N(t)\ \ \mbox{strongly in}\ \ L^2(\Gamma_0)
$$
for a.e. $t\in \br^+$. By Lemma \ref{lemma-usc} we are in position to invoke the convergence theorem of Aubin and Cellina (see for example Theorem 1.4.1 in \cite{aubin-cellina}) and deduce that
$$
\xi(t)\in \partial j(v_N(t)).
$$
Hence we can pass to the limit in the term $(\xi_n(t),z_N)_{L^2(\Gamma_0)}$. It remains to show that $v$ satisfies (\ref{eq:estimate_v_meaning}).
The argument is similar to that in the proof of Theorem \ref{thm-existence-sol}. We choose $0\leq t_1<t_2$ and $\psi\in C^\infty_0(t_1,t_2)$. We have, after possible refining to a subsequence, $\|v_n(t)\|_H^2\psi'(t)\to \|v(t)\|_H^2\psi'(t)$ for a.e. $t\in (t_1,t_2)$. Moreover, since weakly-* convergent sequences in $L^\infty(t_1,t_2;H)$ are bounded in this space it follows that there exists a majorant for $\|v_n(t)\|_H^2\psi'(t)$.
Hence, by the Lebesgue dominated convergence theorem and since $v_n(t)(\psi(t))^{\frac{1}{2}}\to v(t)(\psi(t))^{\frac{1}{2}}$ weakly in $L^2(t_1,t_2;V)$, we can pass to the limit in (\ref{eq:estimate_v_meaning}) written for $v_n$, which completes the proof.
\end{proof}
\begin{proof} \textit{(of Theorem \ref{thm:w_trajectory_attractor})}
Observe that from Lemma \ref{lemma:estimates_b} it follows that for every $v\in \C{K}^+$ we have $\|v\|_{\C{F}^b_+} < \infty$.
Note that since the problem is autonomous, for any $v\in \C{K}^{+}$ and any $h\in\br^+$, the function $v_h$, defined as $v_h(t)=v(t+h)$ for all $t\geq 0$,
belongs to $\C{K}^+$ and hence (\ref{eqn:semiinvariance}) holds.

Now let us estimate $\|T(s)v\|_{\C{F}^b_+}$ for $v\in \C{K}^+$ and $s\geq 0$. We have,
$$
\|T(s)v\|_{\C{F}^b_+} = \sup_{h\geq s}\left\{\sqrt{\int^{h+1}_h\|v(t)\|^2\ dt}+
||v(t)||_{L^\infty(h,h+1;H)}
+\left(\int^{h+1}_h\|v'(t)\|_{V^*}^\frac{4}{3}\ dt\right)^{\frac{3}{4}}\right\}.
$$
Using Lemma \ref{lemma:estimates_b} we get
$$
\|T(s)v\|_{\C{F}^b_+} \leq \sup_{h\geq s}\left\{2\sqrt{C_4+C_5\|v\|_{L^\infty(0,1;H)}^2e^{-C_6h}}
+\left(C_7+C_8\|v\|_{L^\infty(0,1;H)}^\frac{8}{3} e^{-C_9h}\right)^{\frac{3}{4}}\right\}.
$$
By a simple calculation we obtain, for all $s\in \br^+$,
\begin{equation}\label{eqn:absorbing}
\|T(s)v\|_{\C{F}^b_+} \leq R_0+C\|v\|_{L^\infty(0,1;H)}^\beta e^{-\delta s} \leq R_0+C\|v\|_{\C{F}(0,1)}^\beta e^{-\delta s},
\end{equation}
where $C,R_0,\beta,\delta>0$ do not depend on $s,v$.

Now we define the set $\C{P}$ as
$$
\C{P}=\{v\in \C{K}^+: \|v\|_{\C{F}^b_+}\leq 2R_0\}.
$$
We will show that $\C{P}$ is absorbing (and hence also attracting) for $\{T(t)\}$. Let $\mathfrak{B}$ be bounded in $\C{F}^b_+$.
Hence $\mathfrak{B}$ is bounded in $\C{F}(0,1)$. Let $\|v\|_{\C{F}(0,1)}\leq R$ for $v\in \mathfrak{B}$.
Now we choose $s_0>0$ such that $C R^\beta e^{-\delta s_0}\leq R_0$.
From (\ref{eqn:absorbing}) we have for $s\geq s_0$ that
$$
\|T(s)v\|_{\C{F}^b_+} \leq 2R_0,
$$
and we deduce that $\C{P}$ is absorbing.

It suffices to show that $\C{P}$ is compact in the topology $\Theta^{loc}_+$. The fact that it is relatively compact follows from its boundedness
and basic properties of weak compactness in reflexive Banach spaces. It remains to show that it is closed. Let $v_n\stackrel{\Theta^{loc}_+}{\to} v$ and
$\{v_n\}\subset \C{P}$. From the weak lower semicontinuity of the norm it follows that
$$
\|v\|_{\C{F}^b_+} \leq \liminf_{n\to\infty}\|v_n\|_{\C{F}^b_+}\leq 2R_0.
$$
Moreover from Lemma \ref{lemma:closed} it follows that $v\in \C{K}^+$. Thus $v\in \C{P}$ and the proof is complete.
\end{proof}

Now we show that any section of the trajectory attractor is a weak global attractor. We will start from the definition of a weak global attractor (c.f., e.g., \cite{chep-vish-2011}).
\begin{definition}
A set $\C{A}\subset H$ is called a \textit{weak global attractor} of the set $\C{K}^+$ if it is weakly compact in $H$
and the following properties hold
\begin{itemize}
\item[$(i)$] for any set $\mathfrak{B}\subset \C{K}^+$ bounded in the norm of $\C{F}^b_+$ its section $\mathfrak{B}(t)$ is attracted
to $\C{A}$ in the weak topology of $H$ as $t\to\infty$, that is, for any neighbourhood $\C{O}_w(\C{A})$ of $\C{A}$ in the weak topology of
$H$ there exists a time $\tau=\tau(\mathfrak{B},\C{O}_w)$ such that
$$
\mathfrak{B}(t)\subset \C{O}_w(\C{A})\ \ \mbox{for all}\ t\geq\tau,
$$
\item[$(ii)$] $\C{A}$ is the minimal weakly closed set in $H$ that attracts the sections of all bounded sets in $\C{K}^+$ in the weak topology
of $H$ as $t\to\infty$.
\end{itemize}
\end{definition}
We prove the following Theorem
\begin{theorem}
The set $\C{A}\subset H$ defined as $\C{A}=\mathfrak{U}(0)$ is a weak global attractor of $\C{K}^+$.
\end{theorem}
\begin{proof} The argument is standard in the theory of trajectory attractors and follows the lines of the proof of Assertion 5.3 in \cite{chep-vish-2011}. Since $\mathfrak{U}$ is bounded in $\C{F}^b_+$ and thus in $\C{F}(0,1)$, by (\ref{eq:lions-magenes}) we deduce that
$\mathfrak{U}(0)$ is bounded in $H$. Moreover, since the topology $\Theta^{loc}_+$ is stronger that the topology of $C_w([0,T];H)$ it follows that if $u_n\to u$ in the topology $\Theta^{loc}_+$ then $u_n(0)\to u(0)$ weakly in $H$. From the fact that $\mathfrak{U}$ is compact it follows that $\mathfrak{U}(0)$ is weakly compact in $H$.

To show $(i)$ let us take $\mathfrak{B}\subset \C{K}^+$ bounded in the norm of $\C{F}^b_+$. This set is attracted to $\mathfrak{U}$ in the topology $\Theta^{loc}_+$. Since for every sequence $u_n\to u$ in the topology $\Theta^{loc}_+$ we have $u_n(0)\to u(0)$ weakly in $H$,
it follows that $(T(t)\mathfrak{B})(0)$ is attracted to $\mathfrak{U}(0)$ in the weak topology of $H$, or, in other words, $\mathfrak{B}(t)$ is attracted to $\mathfrak{U}(0)$ in the weak topology of $H$.

To show $(ii)$ let $\bar{\C{A}}$ be an arbitrary weakly closed subset of $H$ that weakly attracts the sections of bounded sets $\mathfrak{B}\subset \C{K}^+$, that is, for any weak neighbourhood $\C{O}_w(\bar{\C{A}})$ we have $\mathfrak{B}(t)\subset \C{O}_w(\bar{\C{A}})$ for all $t\geq \tau=\tau(\C{O}_w,\mathfrak{B})$. Let us take $\mathfrak{B}=\mathfrak{U}$. We have $\mathfrak{U}(t)\subset \C{O}_w(\bar{\C{A}})$ for all $t\geq \tau=\tau(\C{O}_w)$. From the strict invariance of the trajectory attractor we obtain that $\C{A} = \mathfrak{U}(0)= \mathfrak{U}(t)\subset \C{O}_w(\bar{\C{A}})$. Since $\bar{\C{A}}$ is  weakly closed in $H$ it follows
that $\C{A}\subset \bar{\C{A}}$ and the proof is complete.
\end{proof}
\renewcommand{\theequation}{\arabic{section}.\arabic{equation}}
\setcounter{equation}{0}
\section{Conclusions} \label{conclusions}
Motivated by some control problems for fluid flow in domains with semipermeable walls and membranes appearing in the lubrication problems, we proved the existence of global in time solutions for a two-dimensional nonstationary Navier-Stokes shear flow with multivalued and nonmonotone boundary conditions on a part of the boundary of the flow domain. The problem is governed by a partial differential inclusion with a multivalued term in the form of Clarke subdifferential. Our main goal was to prove the existence of a trajectory attractor and a weak global attractor for the associated multivalued semiflow.


\end{document}